\documentclass[10pt]{amsart}
\usepackage{amsmath}
\usepackage{amssymb}
\usepackage{amsthm}
\usepackage[final]{hyperref}
\hypersetup{colorlinks=true, linkcolor=blue, anchorcolor=blue, citecolor=red, filecolor=blue, menucolor=blue, pagecolor=blue, urlcolor=blue}
\usepackage{graphicx}
\newtheorem{thm}{Theorem}
\newtheorem{lem}[thm]{Lemma}
\newtheorem{cor}[thm]{Corollary}
\newtheorem{prop}[thm]{Proposition}
\theoremstyle{definition}

\theoremstyle{definition}

\newcommand{\x}[2]{X_{\pm}^{#1, #2}}

\newcommand{\mr}[1]{\mathbb{R}^{#1}}
\newcommand{\apm}{A_{\mu, \pm}}
\newcommand{\ppm}{\phi_{\pm}}
\newcommand{\lt}{L_{t}^{\infty}}
\newcommand{\hx}[1]{H_{x}^{#1}}

\title[Unconditional Uniqueness for Chern-Simons-Higgs]{A Remark on Unconditional Uniqueness in the Chern-Simons-Higgs Model}

\author{Sigmund Selberg}
\address{Department of Mathematical Sciences\\
Norwegian University of Science and Technology\\
Alfred Getz' vei 1\\
N-7491 Trondheim\\ Norway}
\email{sigmund.selberg@math.ntnu.no}

\author{Daniel Oliveira da Silva}
\address{Department of Mathematical Sciences\\
Norwegian University of Science and Technology\\
Alfred Getz' vei 1\\
N-7491 Trondheim\\ Norway}
\email{daniel.dasilva@math.ntnu.no}

\keywords{Chern-Simons-Higgs; well-posedness; unconditional uniqueness; Lorenz gauge; null forms}
\subjclass[2000]{35Q40; 35L70; 81V10}

\thanks{Both authors were supported by the Research Council of Norway, grant no.~213474/F20.}

\begin{document}

\begin{abstract}
The solution of the Chern-Simons-Higgs model in Lorenz gauge with data for the potential in $H^{s-1/2}$ and for the Higgs field in $H^s \times H^{s-1}$ is shown to be unique in the natural space $C([0,T];H^{s-1/2} \times H^s \times H^{s-1})$ for $s \ge 1$, where $s=1$ corresponds to finite energy. Huh and Oh recently proved local well-posedness for $s > 3/4$, but uniqueness was obtained only in a proper subspace $Y^s$ of Bourgain type. We prove that any solution in $C([0,T];H^{1/2} \times H^1 \times L^2)$ must in fact belong to the space $Y^{3/4+\epsilon}$, hence it is the unique solution obtained by Huh and Oh.
\end{abstract}

\maketitle


\section{Introduction}
The (2+1)-dimensional Chern-Simons-Higgs model, first studied by Hong, Kim, and Pac \cite{HKP1990} and Jackiw and Weinberg \cite{JW1990}, consists of a Higgs field $\phi : \mathbb{R}^{1+2} \rightarrow \mathbb{C}$ and a gauge field $A : \mr{1+2} \longrightarrow \mr{1+2}$ which satisfy the equations
\begin{equation}\label{CSH}
\begin{aligned}
& D_\mu D^\mu \phi = - \phi V'\left( |\phi |^2\right) \\
& F_{\mu \nu} = \frac{2}{\kappa} \epsilon_{\mu \nu \rho} \textrm{Im}\left( \overline{\phi} D^\rho \phi \right).
\end{aligned}
\end{equation}
Here, $D_\mu = \partial_\mu - i A_\mu$ denotes the covariant derivative associated to $A$, $V$ is a given potential which will be assumed to be a polynomial, $\kappa$ is a constant, $F_{\mu \nu} = \partial_\mu A_\nu - \partial_\nu A_\mu$ is the curvature, and $\epsilon_{\mu \nu \rho}$ is the skew-symmetric tensor with $\epsilon_{0 1 2} = 1$.  The Einstein summation convention is in effect, and we raise and lower indices with the Minkowski metric $g_{\mu \nu} = \textrm{diag}(1,-1,-1)$.

We would like to focus our discussion on the subject of local well-posedness of the Cauchy problem for these equations, with initial condition
\[
  A_\mu(0) = a_\mu \quad (\mu = 0,1,2), \qquad (\phi,\partial_t\phi)(0) = (f,g).
\]
Currently, the best result on the matter is that of Huh and Oh \cite{HO2012}, in which local well-posedness for $H^{\frac{3}{4}+\epsilon}$ initial data sets was shown, improving earlier results in \cite{H2007, B2009, H2011, ST2013}.

Here $H^s = (1-\Delta)^{-s/2} L^2(\mr{2})$, and $(a_\mu,f,g)$ is said to be an \emph{$H^s$ initial data set} if $a_\mu \in H^{s-1/2}$, $f \in H^s$, $g \in H^{s-1}$, and the constraint equation \[\partial_1 a_2 - \partial_2 a_1 = \frac{2}{\kappa} \textrm{Im}\left( \overline{f} (g-ia_0 f) \right)\] is satisfied.

In the context of well-posedness for the Chern-Simons-Higgs model, two values of $s$ are of particular importance: $s=1/2$ is the scale invariant regularity, and $s=1$ is the energy regularity. Well-posedness for $s=1$ was first proved in \cite{ST2013}, where it was also shown that solutions with finite-energy data extend globally in time. The global result for more regular data ($s \ge 2$) was proved earlier, in \cite{CC2002}.

The proof of Huh and Oh involves a contraction argument in a proper subspace $Y^s$ of the natural evolution space $C([0,T] ; H^{s-1/2} \times H^s \times H^{s-1})$. Thus, the argument of Huh and Oh yields only \emph{conditional well-posedness}, in the sense that the uniqueness is known only in the contraction space $Y^s$.  The goal of the present work is to improve this to an unconditional result for $H^s$ initial data sets with $s \ge 1$.  This is the content of the following theorem.

\begin{thm}\label{mainthm}
Let $s \geq 1$.  Consider \eqref{CSH} complemented by the Lorenz gauge condition $\partial^\alpha A_\alpha = 0$. Given an $H^s$ data set $(a_\mu,f,g)$ and a time interval $I$ containing $t=0$, there is at most one solution $(A_\mu,\phi,\partial_t\phi)$ in $C(I; H^{s-1/2} \times H^s \times H^{s-1})$ with the prescribed initial data.  Thus, \eqref{CSH} is unconditionally locally well posed.
\end{thm}

In particular, this result includes the energy regularity $s=1$, and this was the main motivation for the present work. Our proof can very likely be pushed further to prove the same result even for some range of $s$ below energy (by using space-time product estimates from \cite{DFS2010}), but we do not consider this problem here.

The proof of Theorem \ref{mainthm} is based on an idea of Zhou \cite{Zhou2000} (see also \cite{ST2013b}).  We prove that any solution residing in the energy space $C(I; H^{1/2} \times H^1 \times L^2)$ must belong to some rougher spaces of $X^{s,b}$ type. Iterating the argument a finite number of times we are then able to show that the solution belongs to the low-regularity space $Y^{\frac34 + \epsilon}$ in which Huh and Oh proved uniqueness. In the next section, we will begin the proof of Theorem \ref{mainthm} by introducing some simplifications, as well as some important estimates that will be needed.  The details of the proof are given in sections 3--5.


\section{Preliminaries}\label{prel}

As an initial simplification, we note that we need only prove Theorem \ref{mainthm} for $s = 1$. Next, we introduce the half-wave decompositions
\begin{equation}\begin{aligned}\label{halfwave}
A_{\mu, \pm} & = \frac{1}{2}\left( A_\mu \pm \frac{1}{i}\langle \nabla \rangle^{-1} \partial_t A_\mu \right) \\
\phi_{\pm} & = \frac{1}{2}\left( \phi \pm \frac{1}{i}\langle \nabla \rangle^{-1} \partial_t \phi \right).
\end{aligned}\end{equation}
Equivalently,
\begin{equation}\begin{alignedat}{2}\label{halfwavereversed}
  A_\mu &= A_{\mu,+} + A_{\mu,-},& \qquad \partial_t A_\mu &= i \langle \nabla \rangle \left( A_{\mu,+} - A_{\mu,-} \right),
  \\
  \phi &= \phi_+ + \phi_-,& \qquad \partial_t \phi &= i \langle \nabla \rangle \left( \phi_+ - \phi_- \right).
\end{alignedat}
\end{equation}
This decomposition, when combined with the Lorenz gauge condition $\partial^\mu A_{\mu} = 0$, allows us to rewrite \eqref{CSH} in the form
\begin{equation}\label{CSH2}
\begin{aligned}
\left(i \partial_t \pm \langle \nabla \rangle\right)A_{\mu, \pm} & = \pm\frac{1}{2}\langle \nabla \rangle^{-1}M_{\mu}(A,\phi) \\
\left(i \partial_t \pm \langle \nabla \rangle\right)\phi_{\pm} & = \pm\frac{1}{2}\langle \nabla \rangle^{-1}N(A,\phi),
\end{aligned}
\end{equation}
where $M_{\mu}$ and $N$ are given by
\begin{displaymath}
\begin{aligned}
M_{\mu}(A,\phi) & = -\epsilon_{\mu\nu\rho} \textrm{Im}\ Q^{\nu\rho} \left(\overline{\phi}, \phi \right) + 2 \epsilon_{\mu\nu\rho} \partial^{\nu} \left( A^\rho | \phi |^2 \right) + A_\mu, \\
N(A,\phi) & = 2iA_\mu \partial^\mu \phi + A_\mu A^\mu \phi - \phi V'(|Ê\phi |^2) + \phi,
\end{aligned}
\end{displaymath}
and $Q_{\mu \nu}(u,v) = \partial_\mu u \partial_\nu v - \partial_\nu u \partial_\mu v$ are Klainerman's null forms.  In $M_\mu$ and $N$ it is understood that we use the substitutions in \eqref{halfwavereversed} to express everything in terms of $A_{\mu,\pm}$ and $\phi_{\pm}$.  We remark that the initial values of $\partial_t A_\mu$ are determined by those for $(A_\mu,\phi,\partial_t\phi)$ via the second equation in \eqref{CSH} and the Lorenz gauge condition. Thus, we consider \eqref{CSH2} for initial data $A_{\mu, \pm}(0) = a_{\mu, \pm} \in H^{1/2}$ and $\phi_{\pm}(0) = f_{\pm} \in H^{1}$.

The appearance of the term $Q_{\mu \nu}$ is noteworthy in that it is an example of a \emph{null form}.  These bilinear structures are known to have many useful properties (see, for example, \cite{SK2002}).  As such, we would like to introduce additional null forms into the nonlinearities, so that these useful properties can be exploited.  Following \cite{ST2013}, we define
\begin{equation}\begin{aligned}\label{defs}
& B_1 = \sum_{\pm_1, \pm_2} \big(  A_{0, \pm_1} \langle \nabla \rangle (\pm_2 \phi_{\pm_2}) - R_j (\pm_1 A_{0, \pm_1}) \partial^j \phi_{\pm_2}\big) \\
& B_2 = R_2 \psi \partial_1 \phi - R_1 \psi \partial_2 \phi \\
& B_3 = \langle \nabla \rangle^{-2}A_{\mu} \partial^{\mu} \phi,
\end{aligned}\end{equation}
where
\[
  \psi = R_1 A_2 - R_2 A_1, \qquad R_j = (1 - \Delta)^{-1/2} \partial_j.
\]
It has been shown that $B_1$ and $B_2$ are null forms \cite{KM1994,ST2010}.  With these definitions, and using the Lorenz gauge condition, we can rewrite $N$ as
\[
N(A, \phi) = 2i ( B_1 - B_2 - B_3 ) + A_{\mu} A^{\mu} \phi - \phi V'\left( |\phi|^2 \right) + \phi.
\]

Next, we will introduce the relevant function spaces.  The most important spaces we will work with are the $X_\pm^{s,b}$ spaces, defined as the closure of the Schwartz space $\mathcal{S}$ under the norm
\[
\| u \|_{\x{s}{b}} = \| \langle \xi \rangle^{s} \langle - \tau \pm |\xi| \rangle^{b} \widetilde{u}(\tau, \xi) \|_{L^2_{\tau, \xi}},
\]
where $\widetilde u(\tau,\xi)$ denotes the space-time Fourier transform of $u(t,x)$ and
$$
  \langle \xi \rangle = (1 + |\xi|^2)^{1/2}.
$$
Observe that
\[
  X_\pm^{s,0} = L_t^2H_x^s
\]
for all $s \in \mathbb R$.

Due to the local nature of our result, we work with local versions of the $X^{s,b}$ spaces.  If $I$ is an interval, we define the restriction space $\x{s}{b}(I)$ by
\[
\| u \|_{\x{s}{b}(I)} = \inf \left\{ \| v \|_{\x{s}{b}} \Big|\ v = u\ \text{on}\ I \times \mr{2} \right\}.
\]
We recall the fact (see, e.g., \cite[section 2.6]{T2006}) that for $b > 1/2$, $\x{s}{b}(I)$ embeds into $C(I;H^s)$. We will work mostly with the restriction spaces, so we usually omit $I$ from the notation for these spaces. 

As a final step to setting up the proof of Theorem \ref{mainthm}, we recall the result of Huh and Oh.
\begin{thm}[\cite{HO2012}]
Suppose that $(a_\mu, f, g)$ is an $H^{3/4 + \epsilon}$ initial data set, where $0 < \epsilon \ll 1$.  Then there exists 
\[
T = T\left(\| a_\mu \|_{H^{1/4 + \epsilon}}, \| (f,g) \|_{H^{3/4 + \epsilon} \times H^{1/4 + \epsilon}}\right) > 0
\]
such that a solution $(A_\mu , \phi)$ to \eqref{CSH} under the Lorenz gauge condition $\partial^\mu A_\mu = 0$ and with the prescribed initial data exists on the time interval $I = (-T, T)$ and satisfies $A_\mu \in C(I; H^{1/4 + \epsilon})$ and $\phi \in C(I; H^{3/4 + \epsilon}) \cap C(I; H^{1/4 + \epsilon})$.
\end{thm}
  We note that in the proof, Huh and Oh obtain solutions $\apm$ and $\ppm$ to \eqref{CSH2} which lie in $\x{1/4 + \epsilon}{3/4 - 2\epsilon}(I)$ and $\x{3/4 + \epsilon}{3/4 - 2\epsilon}(I)$, respectively, and that the solutions are unique in these spaces.  From this, we see that Theorem \ref{mainthm} will follow from the following proposition.
\begin{prop}\label{mainthm2}
Let $\ppm \in C(I; H^1)$ and $\apm \in C(I; H^{1/2})$ be solutions to \eqref{CSH2}.  Then $\apm \in \x{1/4 + \epsilon}{3/4 - \epsilon}(I)$ and $\ppm \in \x{3/4 + \epsilon}{3/4 - \epsilon}(I)$ for $0 < \epsilon \ll 1$.
\end{prop}
  Subsequent sections of this paper are devoted to proving Proposition \ref{mainthm2}.

Finally, we introduce the necessary estimates which are used repeatedly throughout the proof.  First we recall some well-known facts about solvability in $X^{s, b}_\pm(I)$ of the Cauchy problem
\begin{equation}\label{LinearProblem}
  \left(i\partial_t \pm \langle D \rangle \right) u = F \quad \text{in $I \times \mr{2}$}, \qquad u(0,x) = f(x),
\end{equation}
for given $f(x)$, $F(t,x)$, and a time interval $I$ containing $t=0$. A proof of the following result can be found, for example, in \cite[section 2.6]{T2006}.

\begin{lem}\label{energy}
Let $s \in \mathbb{R}$, $1/2 < b \leq 1$, $0 < T \leq 1$, and $I=(-T,T)$.  Assume $F \in X_\pm^{s, b - 1}(I)$ and $f \in H^s$. Then \eqref{LinearProblem} has a solution $u \in X^{s, b}_\pm(I)$ satisfying
\begin{equation}\label{energyest}
\| u \|_{X^{s, b}_\pm(I)} \lesssim \| f \|_{H^s} + \| F \|_{X_\pm^{s, b-1}(I)}.
\end{equation}
\end{lem}

Since $X^{s, b}_\pm(I) \subset C(I;H^s)$ for $b > 1/2$, and since uniqueness holds for \eqref{LinearProblem} in the latter space (that is, if $u \in C(I;H^s)$ satisfies \eqref{LinearProblem} with $(f,F) = (0,0)$, then $u=0$), we conclude that the following holds.

\begin{cor}\label{EnergyCorollary}
Under the hypotheses of Lemma \ref{energy}, if $u \in C(I;H^s)$ satisfies \eqref{LinearProblem}, then $u \in X^{s, b}_\pm(I)$ and satisfies \eqref{energyest}.
\end{cor}

As we see from the last two results, it will be necessary to estimate the norms of the nonlinearities $M_{\mu}$ and $N$.  For this we will rely on the well-known Sobolev product estimate, which we now state.

\begin{lem}\label{prodsob}
If $s_0, s_1, s_2 \in \mathbb{R}$, then the 2d product estimate
\begin{equation}
\| fg \|_{H^{-s_0}} \lesssim \| f \|_{H^{s_1}} \| g \|_{H^{s_2}} 
\end{equation} 
holds if and only if $s_0$, $s_1$, and $s_2$ satisfy
\[\begin{aligned}
s_0 + s_1 + s_2 & \geq 1 \\
s_0 + s_1 + s_2 & \geq \max (s_0, s_1, s_2),
\end{aligned}\]
where at most one of the inequalities above can be an equality.
\end{lem}

We remark that this lemma shows that the terms appearing in the right-hand side of \eqref{CSH2} make sense as distributions if $\ppm \in C(I ; H^{1})$ and $\apm \in C(I ; H^{1/2})$.

Applying the product law twice, one obtains the following trilinear estimates, which will be used to handle the cubic terms in $M_\mu$ and $N$.

\begin{cor}\label{prodsobcor}
The following 2d product estimates hold for any $\epsilon > 0$:
\begin{align*}
\| fgh \|_{H^{-\frac12-\epsilon}} &\lesssim \| f \|_{H^{-\frac12}} \| g \|_{H^{1}} \| h \|_{H^{1}},
\\
\| fgh \|_{H^{-\frac12-\epsilon}} &\lesssim \| f \|_{H^{\frac12}} \| g \|_{L^2} \| h \|_{H^{1}}
\\
\| fgh \|_{H^{-\epsilon}} &\lesssim \| f \|_{H^{\frac12}} \| g \|_{H^{\frac12}} \| h \|_{H^{1}}.
\end{align*}
\end{cor}


\section{Basic Estimates for $\apm$ and $\ppm$}\label{basic}

Assume that $\ppm \in C(I ; H^{1})$ and $\apm \in C(I ; H^{1/2})$ satisfy \eqref{CSH2}.  As a first step towards proving Proposition \ref{mainthm2}, we first prove that $\ppm$ and $\apm$ belong to some rough spaces.  Using Corollary \ref{EnergyCorollary}, we will use these rough estimates to increase the regularity later.  We begin with the following lemma.

\begin{lem}\label{cont}
If $u \in C(I ; H^{s})$, then $u \in \x{s}{0}(I)$ and
\[
\| u \|_{\x{s}{0}(I)} \leq | I |^{1/2} \| u \|_{L^{\infty}_t H^s_x (I \times \mr{2})}.
\]
\end{lem}
\begin{proof}
This follows from H\"older's inequality in $t$, since $\x{s}{0}(I) = L^2_t (I; H^s_x )$.
\end{proof}
  This lemma immediately implies our initial estimate for $\ppm$ and $\apm$.
\begin{prop}\label{inclusion}
If $\ppm \in C(I ; H^{1})$ and $\apm \in C(I ; H^{1/2})$, then $\ppm \in \x{1}{0}$ and $\apm \in \x{\frac{1}{2}}{0}$.
\end{prop}

As a second step, we would like to determine values $s_1$ and $s_2$ for which $\ppm \in \x{s_1}{1}$ and $\apm \in \x{s_2}{1}$.  These will be useful for interpolating with the previous results, leading to a range of estimates which will be used below.  
%

\begin{prop}\label{est1}
If $\ppm \in C(I ; H^{1})$ and $\apm \in C(I ; H^{1/2})$ are solutions to \eqref{CSH2}, then $\ppm \in \x{\frac{1}{2}}{1}$ and $\apm \in \x{-\epsilon}{1}$ for all $\epsilon > 0$.
\end{prop}

\begin{proof}
By Corollary \ref{EnergyCorollary}, we need only show that $\| NÊ\|_{\x{-1/2}{0}}$ and $\| M_{\mu}Ê\|_{\x{-\epsilon-1}{0}}$ are bounded.  As was seen in Lemma \ref{cont}, it suffices to show that $N$ and $M_{\mu}$ are bounded in $\lt \hx{-1/2}$ and $\lt \hx{- \epsilon - 1}$, respectively.  Starting with $M_{\mu}$, recall that
\[
M_{\mu} = -\epsilon_{\mu\nu\rho} \textrm{Im}\ Q^{\nu\rho} \left(\overline{\phi}, \phi \right) + 2 \epsilon_{\mu\nu\rho} \partial^{\nu} \left( A^\rho | \phi |^2 \right) + A_\mu.
\]
Apply Lemma \ref{prodsob} and Corollary \ref{prodsobcor}, respectively, to the first two terms, we obtain
\[\begin{aligned}
\left\| Q_{\nu \rho}\left( \overline{\phi}, \phi \right) \right\|_{\lt \hx{- \epsilon - 1}} & \lesssim \| \phi \|^{2}_{\lt \hx{1}} \\
\left\| \partial_{\nu} \left( A_\rho | \phi |^2 \right) \right\|_{\lt \hx{- \epsilon - 1}} & \lesssim \| A_{\rho} \|_{\lt \hx{1/2}} \| \phi \|^{2}_{\lt \hx{1}}.
\end{aligned}\]
In addition, we have the trivial estimate
\[
\| A_{\mu} \|_{\lt \hx{- \epsilon - 1}} \lesssim \| A_{\mu} \|_{\lt \hx{1/2}}.
\]
It follows that $\| M_{\mu}Ê\|_{\x{-\epsilon-1}{0}} < \infty$.

Next, recall that $N = 2i ( B_1 - B_2 - B_3 ) + A_{\mu} A^{\mu} \phi - \phi V'\left( |\phi|^2 \right) + \phi$, with $B_i$ as defined in \eqref{defs}.  Using the fact that
\[
\| R_{i} \psi \|_{\lt \hx{s}} \lesssim \sum_{\mu} \| A_{\mu} \|_{\lt \hx{s}},
\]
we can apply Lemma \ref{prodsob} to obtain
\[
\| B_i \|_{\lt \hx{-1/2}} \lesssim \left( \sum_{\mu} \| A_{\mu} \|_{\lt \hx{1/2}} \right) \| \phi \|_{\lt \hx{1}} 
\]
for $i = 1, 2, 3$.  By Corollary \ref{prodsobcor},
\[
\| A_\mu A^\mu \phi \|_{\lt \hx{-1/2}} \lesssim \left( \sum_{\mu} \| A_{\mu} \|^2_{\lt \hx{1/2}} \right) \| \phi \|_{\lt \hx{1}}. 
\]
For the final term in $N$, we have the trivial observation
\[
\| \phi \|_{\lt \hx{-1/2}} \lesssim \| \phi \|_{\lt \hx{1}}.
\]

We are left with the task of estimating the term $\phi V'(|Ê\phi |^2)$.  Recall that $V$ is a polynomial.  Thus, it suffices to consider terms of the form $\phi | \phi |^{2k}$ for integer $k$.  Applying Lemma \ref{prodsob} yields
\[
\| \phi | \phi |^{2k} \|_{\lt \hx{-\epsilon}} \lesssim \| \phi \|_{\lt \hx{1}} \left\| | \phi |^{2k} \right\|_{\lt L^2_x}.
\]
If we make the observation that
\[
\left\| | \phi |^{2k} \right\|_{\lt L^2_x} = \| \phi \|_{\lt L^{4k}_x}^{2k},
\]
we can then obtain the estimate
\[\begin{aligned}
\| \phi | \phi |^{2k} \|_{\lt \hx{-\epsilon}} & \lesssim \| \phi \|_{\lt \hx{1}} \left\| \phi \right\|^{2k}_{\lt L^{4k}_x} \\
						  & \lesssim \| \phi \|_{\lt \hx{1}} \left\| \phi \right\|^{2k}_{\lt \hx{1}}
\end{aligned}\]
by Sobolev embedding, for all $k \geq 0$.  It follows that 
\begin{equation}\label{2k}
\left\| \phi V'\left(| \phi |^{2}\right) \right\|_{\lt \hx{-\epsilon}} \lesssim \| \phi \|_{\lt \hx{1}} V'\left(\left\| \phi \right\|^{2}_{\lt \hx{1}}\right).
\end{equation}
Combining all the estimates, we can conclude that
\[
\| N \|_{\lt \hx{-1/2}} < \infty.
\]

\end{proof}

Interpolating the results from Propositions \ref{inclusion} and \ref{est1}, we obtain the following estimate, which is the primary estimate from this section.

\begin{prop}\label{interpolate}
If $\ppm \in C(I; H^1)$, $\apm \in C(I; H^{1/2})$ are solutions to \eqref{CSH2}, then
\begin{align*}
\ppm & \in X_\pm^{\theta + \frac{1}{2}(1-\theta),1 - \theta}(I), \\
\apm & \in X_\pm^{\frac{\theta}{2} - \delta(1-\theta),1 - \theta}(I),
\end{align*}
for all $\theta \in [0,1]$ and $\delta > 0$.
\end{prop}


\section{Improved Estimates for $\ppm$ and $\apm$} \label{improved}
The results of Proposition \ref{interpolate} are not sufficient for us to apply the result of Huh and Oh to obtain unconditional uniqueness.  However, with these basic estimates, we can iterate the argument used in the previous section to increase the regularity of the solutions, which will get us one step closer to the necessary estimates.  The main result for this section is the following.
\begin{prop}\label{est2}
If $\ppm \in C(I; H^1)$, $\apm \in C(I; H^{1/2})$ are solutions to \eqref{CSH2}, then $\ppm \in X_{\pm}^{\frac{3}{4} - \epsilon, \frac{3}{4}}(I)$ and $\apm \in X_{\pm}^{\frac{1}{4}, \frac{3}{4}}(I)$ for all $\epsilon > 0$.
\end{prop}
  Its proof will require us to make use of the null structure of the nonlinearities.  We do this via the following lemma from \cite{ST2013b}. If $\alpha$ and $\beta$ be non-zero vectors in $\mathbb{R}^2$, we denote by $\theta(\alpha, \beta)$ the angle between them.

\begin{lem}\label{null}
For all signs $(\pm,\pm_1, \pm_2)$, all $a, b, c \in [0,1/2]$, all $\lambda, \eta \in \mathbb{R}$, and all non-zero $\eta,\zeta \in \mathbb R^2$, we have the estimate
\[
\theta(\pm_1\eta,\pm_2\zeta) \lesssim \left( \frac{\langle -(\lambda+\mu) \pm | \eta + \zeta | \rangle}{\min (\langle \eta \rangle, \langle \zeta \rangle)} \right)^a + \left( \frac{\langle - \lambda \pm_1 | \eta | \rangle}{\min (\langle \eta \rangle, \langle \zeta \rangle)} \right)^b + \left( \frac{\langle - \mu \pm_2 | \zeta | \rangle}{\min (\langle \eta \rangle, \langle \zeta \rangle)} \right)^c.
\]
\end{lem}

\begin{proof}[Proof of Proposition \ref{est2}]
Invoking Corollary \eqref{EnergyCorollary} once again, we reduce to bounding 
$\| M_{\mu} \|_{X_\pm^{-\frac{3}{4}, -\frac{1}{4}}}$ and $\| N \|_{X_\pm^{-\frac{1}{4} - \epsilon, -\frac{1}{4}}}$.

We begin by estimating $M_{\mu}$. Recall that $Q_{\nu \rho}(\overline{\phi}, \phi)$ is given by
\[\begin{aligned}
Q_{j k}(\overline{\phi}, \phi) & = \sum_{\pm 1,\, \pm 2} \left(\partial_j \overline{\phi}_{\pm_1} \partial_k \phi_{\pm_2} - \partial_k \overline{\phi}_{\pm_1} \partial_j \phi_{\pm_2} \right), \\
Q_{0 k}(\overline{\phi}, \phi) & = \sum_{\pm 1,\, \pm 2} \left( -i\langle \nabla \rangle \left(\pm_1 \overline{\phi}_{\pm_1}\right) \partial_k \phi_{\pm_2} - \partial_k \overline{\phi}_{\pm_1} i\langle \nabla \rangle \left(\pm_2 \phi_{\pm_2}\right) \right).
\end{aligned}\]
The complex conjugate $\overline{\phi}_{\pm}$ belongs to $X_{\mp}^{s,b}$ whenever $\phi_\pm$ belongs to $X^{s,b}_\pm$ (note the sign reversal). Thus, writing $u=\overline{\phi}_{\pm_1}$ and $v = \phi_{\pm_2}$, it suffices to consider
\begin{gather}
\label{q1}
\| \partial_j (\pm_1 u) \partial_k (\pm_2 v) - \partial_k (\pm_1 u) \partial_j (\pm_2 v) \|_{X_\pm^{-\frac{3}{4}, -\frac{1}{4}}}
\\
\label{q2}
\| \partial_k (\pm_1 u) \langle \nabla \rangle v - \langle \nabla \rangle u \partial_k (\pm_2 v) \|_{X_\pm^{-\frac{3}{4}, -\frac{1}{4}}},
\end{gather}
As in \cite{ST2013}, we can estimate both of these quantities by $\| I(\tau, \xi) \|_{L^2_{\tau, \xi}}$, where
\begin{equation}\label{bigI}
I (\tau, \xi) = \int_{\mr{1+2}} \frac{\sigma (\pm_1 \eta, \pm_2 (\xi - \eta))}{\langle \xi \rangle^{s} \langle |\tau| - |\xi| \rangle^{b}} | \hat{u}(\lambda, \eta) | | \hat{v}(\tau - \lambda, \xi - \eta) |\ d\lambda d\eta,
\end{equation}
for $s = \frac{3}{4}$, $b = \frac{1}{4}$, and $\sigma$ is given by either of the expressions
\begin{equation}\begin{gathered}\label{alphas}
\sigma (\alpha, \beta) = | \alpha \times \beta | \leq | \alpha | | \beta | \theta(\alpha, \beta),
\\
\sigma (\alpha, \beta) = | \langle \alpha \rangle \beta_k - \alpha_j \langle \beta \rangle | \leq | \alpha | | \beta | \theta(\alpha, \beta) + \frac{| \alpha |}{\langle \beta \rangle} + \frac{| \beta |}{\langle \alpha \rangle}.
\end{gathered}
\end{equation}

Applying Lemma \ref{null} with $a = 1/4$, $b = c = 1/2$, we can now reduce the estimates of \eqref{q1} and \eqref{q2} to the following product estimates. We will denote by $\Lambda_\pm^b$ the Fourier multiplier with symbol $\langle -\tau \pm |\xi| \rangle^b$.
\begin{equation}\begin{aligned}\label{s1}
\|D^{\frac{3}{4}}u \cdot Dv\|_{X_\pm^{-\frac{3}{4}, 0}} & \lesssim \| u \|_{\lt \hx{1}} \| v \|_{\lt \hx{1}},
\\
\|D^{1/2}\Lambda_{\pm_1}^{1/2}u \cdot Dv\|_{X_\pm^{-\frac{3}{4}, -\frac{1}{4}}} & \lesssim \| u \|_{X_{\pm_1}^{\frac34, \frac12}} \| v \|_{\lt \hx{1}},
\\
\|D^{1/2}u \cdot D\Lambda_{\pm_2}^{1/2}v\|_{X_\pm^{-\frac{3}{4}, -\frac{1}{4}}} & \lesssim \| u \|_{\lt \hx{1}} \| v \|_{X_{\pm_2}^{\frac{3}{4}, \frac12}},
\\
\|D u \cdot \langle D \rangle^{-1}v\|_{X_\pm^{-\frac{3}{4}, -\frac{1}{4}}} & \lesssim \| u \|_{\lt \hx{1}} \| v \|_{\lt \hx{1}}.
\end{aligned}\end{equation}
Observe that the norms on the right are all among the ones in which $\phi_\pm$ is known to be bounded, in view of Proposition \ref{interpolate}.  Replacing the $X_\pm^{-\frac{3}{4}, -\frac{1}{4}}$ norms on the left by the larger $L_t^2H^{-\frac34} = X_\pm^{-\frac{3}{4}, 0}$ norms, all the estimates are seen to hold by applying Lemma \ref{prodsob} and H\"older's inequality in time. In more detail, to prove the second estimate we write
\begin{align*}
  \|D^{1/2}\Lambda_{\pm_1}^{1/2}u \cdot Dv\|_{X_\pm^{-\frac{3}{4}, -\frac{1}{4}}}
  &\le \|D^{1/2}\Lambda_{\pm_1}^{1/2}u \cdot Dv\|_{L_t^2H^{-\frac34}}
  \\
  & \lesssim \| D^{1/2}\Lambda_{\pm_1}^{1/2}u  \|_{L_t^2 H^{\frac14}} \| Dv \|_{L^\infty_t L^2_x}
  \\
  & \lesssim \| u \|_{X_{\pm_1}^{\frac34, \frac12}} \| v \|_{\lt \hx{1}},
\end{align*}
and we proceed similarly to prove the other estimates in \eqref{s1}.

Applying Corollary \ref{prodsobcor} to the term $\partial_\nu \left( A_\rho | \phi |^2 \right)$ expanded by Leibniz's rule yields
\begin{equation}\label{s2}
\left\| \partial_\nu \left( A_\rho | \phi |^2 \right) \right\|_{\x{-\frac{3}{4}}{-\frac{1}{4}}}
\le
\left\| \partial_\nu \left( A_\rho | \phi |^2 \right) \right\|_{L_t^2 H^{-\frac12-\epsilon}}
 < \infty.
\end{equation}
Finally, we make the trivial observation
\begin{equation}\label{s3}
\| A_{\mu} \|_{X_\pm^{-\frac{3}{4}, -\frac{1}{4}}} \lesssim \| A_{\mu} \|_{\x{1/2}{0}}.
\end{equation}
Based on the results from equations \eqref{s1} to \eqref{s3}, it follows that
\[
\| M_{\mu} \|_{\x{-\frac{3}{4}}{-\frac{1}{4}}} < \infty.
\]

Now we consider $N = 2i ( B_1 - B_2 - B_3 ) + A_{\mu} A^{\mu} \phi - \phi V'\left( |\phi|^2 \right) + \phi$.  Starting with $B_1$, it suffices to consider
\begin{equation}\label{b1}
\left\|  u \langle \nabla \rangle v - R_j (\pm_1 u) \partial^j (\pm_2 v) \right\|_{\x{-\frac{1}{4}-\epsilon}{-\frac{1}{4}}},
\end{equation}
where $u$ and $v$ belong to the same spaces as $A_{\mu,\pm_1}$ and $\phi_{\pm_2}$, respectively.  As was the case for $Q_{\mu \nu}$, we can estimate the above norm by $\| I(\tau, \xi) \|_{L^2_{\tau, \xi}}$, where $I(\tau, \xi)$ is given by \eqref{bigI} with $s = \frac{1}{4}+\epsilon$, $b = \frac{1}{4}$, and
\[
\sigma (\alpha, \beta) = \left| \langle \beta \rangle - \frac{\alpha \cdot \beta}{\langle \alpha \rangle}  \right| \lesssim | \beta | \theta^2(\alpha, \beta) + \langle \beta \rangle \left( \frac{1}{\langle \alpha \rangle^2} + \frac{1}{\langle \beta \rangle^2} \right).
\]
Applying the trivial estimate
\[
  \theta^2(\alpha, \beta) \lesssim \theta(\alpha, \beta),
\]
and using Lemma \ref{null} with $a = 1/4$, $b = c = 1/2$, we can reduce \eqref{b1} to the following estimates, where the norms in which $u$ and $v$ are taken on the right-hand side are among the ones in which $A_{\mu,\pm}$ and $\phi_\pm$, respectively, are bounded. 
\begin{equation}\begin{aligned}\label{t1}
\|\langle D \rangle^{-\frac{1}{4}}u\cdot Dv\|_{X_\pm^{-\frac{1}{4}-\epsilon, 0}} & \lesssim \| u \|_{\lt \hx{\frac{1}{2}}} \| v \|_{\lt \hx{1}},
\\
\|u\cdot D^{\frac{3}{4}}v\|_{X_\pm^{-\frac{1}{4}-\epsilon, 0}} & \lesssim \| u \|_{\lt \hx{\frac{1}{2}}} \| v \|_{\lt \hx{1}},
\\
\|\langle D \rangle^{-1/2}\Lambda_{\pm_1}^{1/2}u\cdot Dv\|_{X_\pm^{-\frac{1}{4}-\epsilon, -\frac{1}{4}}} & \lesssim \| u \|_{X_{\pm_1}^{\frac14-\epsilon,\frac12}} \| v \|_{\lt \hx{1}},
\\
\|\Lambda_{\pm_1}^{1/2}u\cdot D^{1/2}v\|_{X_\pm^{-\frac{1}{4}-\epsilon, -\frac{1}{4}}} & \lesssim \| u \|_{X_{\pm_1}^{\frac{1}{4}-\epsilon,\frac{1}{2}}} \| v \|_{\lt \hx{1}},
\\
\|\langle D \rangle^{-1/2}u\cdot D\Lambda_{\pm_2}^{1/2}v\|_{X_\pm^{-\frac{1}{4}-\epsilon, -\frac{1}{4}}} & \lesssim \| u \|_{\lt \hx{\frac12}} \| v \|_{X_{\pm_2}^{\frac{3}{4},\frac{1}{2}}},
\\
\| u\cdot D^{1/2}\Lambda_{\pm_2}^{1/2}v\|_{X_\pm^{-\frac{1}{4}-\epsilon, -\frac{1}{4}}} & \lesssim \| u \|_{\lt \hx{\frac12}} \| v \|_{X_{\pm_2}^{\frac{3}{4},\frac{1}{2}}},
\\
\| \langle D \rangle^{-2} u \cdot D v \|_{\x{- \frac{1}{4}-\epsilon}{- \frac{1}{4}}} & \lesssim \| u \|_{\lt \hx{\frac{1}{2}}} \| v \|_{\lt \hx{1}}, 
\\
\| u \cdot \langle D \rangle^{-1} v \|_{\x{- \frac{1}{4}-\epsilon}{- \frac{1}{4}}} & \lesssim \| u \|_{\lt \hx{\frac{1}{2}}} \| v \|_{\lt \hx{1}}.
\end{aligned}\end{equation}
Replacing the norms on the left by the larger $L_t^2H^{-\frac14-\epsilon}$ norms, the estimates are proved by applying Lemma \ref{prodsob} and H\"older's inequality in time. Thus,
\[
\| B_1 \|_{\x{- \frac{1}{4}-\epsilon}{- \frac{1}{4}}} < \infty.
\]

For the term $B_2$, we recall that 
\[
B_2 = R_1 \psi \partial_2 \phi - R_2 \psi \partial_1 \phi,
\] 
with $R_i = (1-\Delta)^{-1/2}\partial_i$ and $\psi = R_1 A_2 - R_2 A_2$.  Thus, $\psi$ enjoys the same estimates as $A$, so it is enough to consider
\[
\| R_j (\pm_1 u) \partial_k (\pm_2 v) - R_k (\pm_1 u) \partial_j (\pm_2 v) \|_{X_\pm^{-\frac{1}{4}-\epsilon, -\frac{1}{4}}},
\]
where $u$ and $v$ belong to the same spaces as $A_{\mu,\pm_1}$ and $\phi_{\pm_2}$, respectively.  Once more we can reduce to estimating $\| I(\tau, \xi) \|_{L^2_{\tau, \xi}}$, where $I(\tau, \xi)$ is as in \eqref{bigI} for $s = \frac{1}{4}+\epsilon$, $b = \frac{1}{4}$, and
\[
\sigma (\alpha, \beta) = \left| \frac{\alpha}{\langle \alpha \rangle} \times \beta \right| \leq | \beta | \theta(\alpha, \beta).
\]
From the first six estimates in \eqref{t1} we therefore get
\[
\| B_2 \|_{\x{- \frac{1}{4}-\epsilon}{- \frac{1}{4}}} < \infty.
\]

The remaining terms can be estimated much more simply.  For $B_3 = \langle \nabla \rangle^{-2} A_\mu \partial^\mu \phi$, we apply Lemma \ref{prodsob} to obtain
\begin{equation}\label{t3}
\| \langle \nabla \rangle^{-2} A_\mu \partial^\mu \phi \|_{\x{-\frac{1}{4}-\epsilon}{-\frac{1}{4}}} \lesssim \left( \sum_{\mu} \| A_\mu \|_{\lt \hx{\frac{1}{2}}}\right) \| \phi \|_{\lt \hx{1}}.
\end{equation}
Next, by Corollary \ref{prodsobcor},
\begin{equation}\label{t4}
\| A_{\mu} A^{\mu} \phi \|_{\x{-\frac{1}{4}-\epsilon}{-\frac{1}{4}}} \lesssim \left(\sum_{\mu}\| A_{\mu} \|^2_{\lt \hx{\frac{1}{2}}}\right) \| \phi \|_{\lt \hx{1}}.
\end{equation}
For the term $\phi V'\left(| \phi |^{2} \right)$ we have, recalling \eqref{2k},
\begin{equation}\label{t5}
\left\| \phi V'\left(| \phi |^{2} \right)  \right\|_{\x{-\frac{1}{4}-\epsilon}{-\frac{1}{4}}}
\lesssim
\left\| \phi V'\left(| \phi |^{2} \right)  \right\|_{L^2_t H^{-\epsilon}} 
\lesssim
 \| \phi \|_{\lt \hx{1}} V'\left( \| \phi \|_{\lt \hx{1}}^{2} \right).
\end{equation}
Finally, we have the trivial observation that
\begin{equation}\label{t6}
\| \phi \|_{\x{-\frac{1}{4}-\epsilon}{-\frac{1}{4}}} \lesssim \| \phi \|_{\x{\frac{3}{4}}{\frac{1}{2}}}.
\end{equation}
Combining the results of equations \eqref{t1} to \eqref{t6}, it follows that
\[
\| N \|_{\x{-\frac{1}{4}-\epsilon}{-\frac{1}{4}}} < \infty.
\]
\end{proof}


\section{Final Estimates for $\ppm$ and $\apm$} \label{final}

We note that by interpolating between the results in Propositions \ref{inclusion} and \ref{est2}, we have
\begin{equation}\label{newfact}
  \ppm \in X^{\frac56-\epsilon,\frac12}, \qquad \apm \in X^{\frac13,\frac12},
\end{equation}
for all $\epsilon > 0$. Repeating the arguments from the previous section, but taking into account the improvement \eqref{newfact}, we now show that our solutions belong to the appropriate spaces as in Proposition \ref{mainthm2}, which we restate and prove here.

\begin{prop}\label{est3}
If $\ppm \in C(I; H^1)$ and $\apm \in C(I; H^{1/2})$ are solutions to \eqref{CSH2}, then $\ppm \in \x{\frac{3}{4} + \epsilon}{\frac{3}{4} - \epsilon}(I)$, and $\apm \in \x{\frac{1}{4} + \epsilon}{\frac{3}{4} - \epsilon}(I)$ for $0 < \epsilon \ll 1$.
\end{prop}
\begin{proof}
We once again start by applying Corollary \ref{EnergyCorollary} to reduce the problem to proving that
$\| M_{\mu} \|_{\x{-\frac{3}{4}+\epsilon}{-\frac{1}{4}-\epsilon}}$ and $\| N \|_{\x{-\frac{1}{4}+\epsilon}{-\frac{1}{4}-\epsilon}}$ are bounded. We will only estimate the bilinear terms, since the other terms are easily handled as in the previous section (there was a lot of room in all the estimates for those terms).

We start with the null form in $M_{\mu}$. Proceeding as in the proof of Proposition \ref{est2}, but taking into account that now $s=\frac34-\epsilon$ and $b=\frac14+\epsilon$ in the definition of $I(\tau,\xi)$, so that we may apply Lemma \ref{null} with $a=\frac14 + \epsilon$ and $b=c=\frac12$, then we reduce to the following set of estimates, instead of \eqref{s1}.
\begin{equation}\begin{aligned}\label{s1final}
\|D^{\frac{3}{4}-\epsilon}u \cdot Dv\|_{X_\pm^{-\frac{3}{4}+\epsilon, 0}} & \lesssim \| u \|_{\lt \hx{1}} \| v \|_{\lt \hx{1}},
\\
\|D^{1/2}\Lambda_{\pm_1}^{1/2}u \cdot Dv\|_{X_\pm^{-\frac{3}{4}+\epsilon, -\frac{1}{4}-\epsilon}} & \lesssim \| u \|_{X_{\pm_1}^{\frac56-\epsilon, \frac12}} \| v \|_{\lt \hx{1}},
\\
\|D^{1/2}u \cdot D\Lambda_{\pm_2}^{1/2}v\|_{X_\pm^{-\frac{3}{4}+\epsilon, -\frac{1}{4}-\epsilon}} & \lesssim \| u \|_{\lt \hx{1}} \| v \|_{X_{\pm_2}^{\frac56 - \epsilon, \frac12}},
\\
\|D u \cdot \langle D \rangle^{-1}v\|_{X_\pm^{-\frac{3}{4}+\epsilon, -\frac{1}{4}-\epsilon}} & \lesssim \| u \|_{\lt \hx{1}} \| v \|_{\lt \hx{1}}.
\end{aligned}\end{equation}
Replacing the norms on the left by the larger $L_t^2H^{-\frac34+\epsilon}$ norms, these estimates reduce to Lemma \ref{prodsob} and H\"older's inequality in time. Observe that the norms on the right are all among the ones in which $\phi_\pm$ is known to be bounded, in view of \eqref{newfact}.

It now follows that
\[
\| M_{\mu} \|_{\x{-\frac{3}{4}+\epsilon}{-\frac{1}{4}-\epsilon}} < \infty.
\]

Next we consider the terms $B_i$ in $N$. We only estimate $B_1$ and $B_2$, since $B_3$ is much easier (there is a lot of room in the estimate \eqref{t3}). Proceeding as in the proof of Proposition \ref{est2}, we reduce to the following estimates, which now replace \eqref{t1}.
\begin{equation}\begin{aligned}\label{t1final}
\|\langle D \rangle^{-\frac{1}{4}-\epsilon}u\cdot Dv\|_{X_\pm^{-\frac{1}{4}+\epsilon, 0}} & \lesssim \| u \|_{\lt \hx{\frac{1}{2}}} \| v \|_{\lt \hx{1}},
\\
\|u\cdot D^{\frac{3}{4}-\epsilon}v\|_{X_\pm^{-\frac{1}{4}+\epsilon, 0}} & \lesssim \| u \|_{\lt \hx{\frac{1}{2}}} \| v \|_{\lt \hx{1}},
\\
\|\langle D \rangle^{-1/2}\Lambda_{\pm_1}^{1/2}u\cdot Dv\|_{X_\pm^{-\frac{1}{4}+\epsilon, -\frac{1}{4}-\epsilon}} & \lesssim \| u \|_{X_{\pm_1}^{\frac13,\frac12}} \| v \|_{\lt \hx{1}},
\\
\|\Lambda_{\pm_1}^{1/2}u\cdot D^{1/2}v\|_{X_\pm^{-\frac{1}{4}-\epsilon, -\frac{1}{4}+\epsilon}} & \lesssim \| u \|_{X_{\pm_1}^{\frac13,\frac{1}{2}}} \| v \|_{\lt \hx{1}},
\\
\|\langle D \rangle^{-1/2}u\cdot D\Lambda_{\pm_2}^{1/2}v\|_{X_\pm^{-\frac{1}{4}+\epsilon, -\frac{1}{4}-\epsilon}} & \lesssim \| u \|_{\lt \hx{\frac12}} \| v \|_{X_{\pm_2}^{\frac56-\epsilon,\frac{1}{2}}},
\\
\| u\cdot D^{1/2}\Lambda_{\pm_2}^{1/2}v\|_{X_\pm^{-\frac{1}{4}+\epsilon, -\frac{1}{4}-\epsilon}} & \lesssim \| u \|_{\lt \hx{\frac12}} \| v \|_{X_{\pm_2}^{\frac56-\epsilon,\frac{1}{2}}},
\\
\| \langle D \rangle^{-2} u \cdot D v \|_{\x{- \frac{1}{4}+\epsilon}{- \frac{1}{4}-\epsilon}} & \lesssim \| u \|_{\lt \hx{\frac{1}{2}}} \| v \|_{\lt \hx{1}}, 
\\
\| u \cdot \langle D \rangle^{-1} v \|_{\x{- \frac{1}{4}+\epsilon}{- \frac{1}{4}-\epsilon}} & \lesssim \| u \|_{\lt \hx{\frac{1}{2}}} \| v \|_{\lt \hx{1}}.
\end{aligned}\end{equation}
As before, we dominate the norms on the left by the larger $L_t^2H^{-\frac14+\epsilon}$ norms. All the estimates are then seen to hold by applying Lemma \ref{prodsob} and H\"older's inequality in time. Moreover, the norms in which $u$ and $v$ are taken on the right-hand side are among the ones in which $A_{\mu,\pm}$ and $\phi_\pm$, respectively, are bounded, in view of \eqref{newfact}.

It follows that
\[
\| N \|_{\x{-\frac{1}{4}+\epsilon}{-\frac{1}{4}-\epsilon}} < \infty.
\]

We conclude that $\ppm \in \x{3/4+\epsilon}{3/4-\epsilon}$ and $\apm \in \x{1/4+\epsilon}{3/4-\epsilon}$, and this finishes the proof of Proposition \ref{mainthm2} and hence of our main result, Theorem \ref{mainthm}.
\end{proof}

\bibliographystyle{amsplain}
\bibliography{database}

\end{document}